\documentclass[11pt]{amsart}

\usepackage{epsfig,overpic,comment}
\usepackage[usenames,dvipsnames,svgnames,table]{xcolor}
\usepackage[hyphens]{url}
\usepackage[pagebackref,linktocpage=true,colorlinks=true,linkcolor=Blue,citecolor=BrickRed,urlcolor=RoyalBlue]{hyperref}
\usepackage[msc-links,abbrev]{amsrefs}
\usepackage{amsmath,amsthm,amssymb,marginnote}
\usepackage{enumerate}
\usepackage[T1]{fontenc}

\newtheorem{theorem}{Theorem}[section]

\newtheorem{lemma}[theorem]{Lemma}

\theoremstyle{definition}

\newcommand{\be}{\begin{equation}}
\newcommand{\ee}{\end{equation}}
\newcommand{\ol}{\overline}

\newcommand{\R}{\mathbf{R}}

\newcommand{\C}{\mathcal{C}}
\newcommand{\G}{\mathcal{G}}

\renewcommand{\epsilon}{\varepsilon}
\renewcommand{\S}{\mathbf{S}}
\renewcommand{\tilde}{\widetilde}

\title[Total curvature of convex hypersurfaces]{Total curvature of convex hypersurfaces in Cartan-Hadamard manifolds}

\author{Mohammad Ghomi}
\address{School of Mathematics, Georgia Institute of Technology,
Atlanta, GA 30332}
\email{ghomi@math.gatech.edu}
\urladdr{www.math.gatech.edu/~ghomi}

\author{John Ioannis Stavroulakis}
\address{School of Mathematics, Georgia Institute of Technology,
Atlanta, GA 30332}
\email{jstavroulakis3@gatech.edu}

\date{\today \,(Last Typeset)}

\thanks{The first-named author was supported by NSF grant DMS-2202337.}

\begin{document}

\maketitle

\begin{abstract}
We show that if the curvature of a Cartan-Hadamard $n$-manifold is constant near a convex hypersurface $\Gamma$, then the total Gauss-Kronecker curvature $\mathcal{G}(\Gamma)$ is not less than that of any convex hypersurface nested inside $\Gamma$. 
This extends  Borb\'{e}ly's monotonicity theorem in hyperbolic space.
It follows that $\mathcal{G}(\Gamma)$ is bounded below by the volume of the unit sphere in Euclidean space $\R^n$. If $n=3$ or $\Gamma$ is $\C^{1,1}$, it suffices that the curvature be constant on $\Gamma$.
\end{abstract}
\section{Introduction}

A Cartan-Hadamard manifold $M^n$ is a complete simply connected space with nonpositive curvature. A \emph{convex hypersurface} $\Gamma\subset M$ is the boundary of a compact convex set with interior points. 
An outstanding question in Riemannian geometry  \cite[p. 66]{ballmann-gromov-schroeder}\cite{willmore-saleemi} is whether the total Gauss-Kronecker curvature 
\begin{equation}\label{eq:main}
\mathcal{G}(\Gamma)\geq|\S^{n-1}|,
\end{equation}
 where $|\S^{n-1}|$ is the volume of the unit sphere in Euclidean space $\R^n$. 
 Establishing this inequality would resolve the Cartan-Hadamard conjecture \cite{kleiner1992, ghomi-spruck2022} concerning the extension of the Euclidean isoperimetric inequality to manifolds of nonpositive curvature \cite{aubin1975,gromov1999,burago-zalgaller1988}.
It is known that \eqref{eq:main} holds for geodesic spheres \cite{ghomi-spruck2022}, and Borb\'{e}ly \cite{borbely2002} showed that it holds in hyperbolic space $\textbf{H}^n$. More generally, he proved that if $\Gamma$, $\gamma$ are convex hypersurfaces in $\textbf{H}^n$  with $\gamma$ nested inside $\Gamma$, then $\mathcal{G}(\Gamma)\geq\mathcal{G}(\gamma)$. We refine this monotonicity result as follows:

\begin{theorem}\label{thm:main}
Let $\Gamma$, $\gamma$ be convex hypersurfaces in a Cartan-Hadamard manifold $M^n$, with $\gamma$ nested inside $\Gamma$. Suppose that the curvature $K$ of $M$ is constant on a neighborhood of $\Gamma$.  Then
$
\mathcal{G}(\Gamma)\geq\mathcal{G}(\gamma).
$
If $n=3$ or $\Gamma$ is $\C^{1,1}$, then it suffices to assume that $K$ is constant on $\Gamma$.
\end{theorem}

Letting $\gamma$ in the above theorem be a geodesic sphere yields \eqref{eq:main}. For $n=2$, the above theorem follows quickly from the Gauss-Bonnet theorem, without any assumptions on $K$. For $n\geq 3$, however,  it is essential that $K$ be constant on $\Gamma$ 
due to examples by Dekster \cite{dekster1981}. When $\Gamma$ is smooth, and the constant in Theorem \ref{thm:main} is the supremum of $K$ on the domain $\Omega$ bounded by $\Gamma$, then $K$ is constant on $\Omega$ \cite[Thm. 1.2]{ghomi2025}, which reduces the above result to Borb\'{e}ly's theorem;  see also \cite{ghomi-spruck2023a,schroeder-strake1989a} for similar rigidity results, which extend  ``gap theorems'' of Greene-Wu \cite{greene-wu1982} and Gromov \cite[Sec. 3]{ballmann-gromov-schroeder}. Since the constant in Theorem \ref{thm:main} is arbitrary, not to mention that no regularity is assumed on $\Gamma$, we need to develop another approach.

We prove Theorem \ref{thm:main} via a comparison formula for total curvature of  nested hypersurfaces \cite{ghomi-spruck2022}. This formula expresses the difference 
$\G(\Gamma)-\G(\gamma)$ as an integral over the region between the hypersurfaces, involving components of the Riemann curvature tensor $R$ of $M$ and derivatives of a function $u$ whose level sets interpolate between $\gamma$ and $\Gamma$.
The key step is the choice of $u$, which is built from the distance functions of $\gamma$ and $\Gamma$.
Convexity of $u$ ensures that the principal curvatures of its level sets are nonnegative, which control the sign of the leading terms in the comparison formula. When $K$
is constant near $\Gamma$, the mixed  terms of $R$ vanish, yielding the desired monotonicity. In dimension three, or in any dimension when $\Gamma$ is $\C^{1,1}$, more delicate estimates control these mixed terms when $K$ is constant only along $\Gamma$.

\section{Preliminaries}
Here we gather four lemmas which we need to prove Theorem \ref{thm:main}. Throughout this work $M$ is a Cartan-Hadamard $n$-manifold with sectional curvature $K$ and Riemann curvature tensor $R$. 
\subsection{The comparison formula}
Let $\Gamma$ be a closed $\C^{1,1}$ hypersurface embedded in $M$. The Gauss-Kronecker curvature $GK$ of $\Gamma$ is the determinant of the second fundamental form of $\Gamma$ with respect to the outward normal.
The \emph{total curvature} of $\Gamma$  is given by
$$
\G(\Gamma):=\int_\Gamma GK,
$$
which is well-defined by Rademacher's theorem. 
A \emph{domain} $\Omega\subset M$ is an open set with compact closure $\ol\Omega$. Let $\Omega$ be the domain bounded by $\Gamma$,  and $\gamma$ be another closed embedded $\C^{1,1}$ hypersurface  which bounds a domain $D$ with $\ol D\subset\Omega$. Then we say that $\gamma$ is \emph{nested} inside $\Gamma$. Suppose there exists a $\C^{1,1}$ function $u$ on $\ol\Omega\setminus D$ with $\nabla u\neq 0$ on $\Omega\setminus \ol D$, which is constant on $\gamma$ and $\Gamma$. We assume that $u|_\gamma<u|_\Gamma$ so that $e_n:=\nabla u/|\nabla u|$ points outward along the level sets of $u$. Let $\kappa_i$ be principal curvatures of the level sets  with respect to $e_n$, and let  $e_1,\dots, e_{n-1}$ form an orthonormal set of the corresponding principal directions.  The \emph{comparison formula}, first proved in \cite{ghomi-spruck2022} and developed further in \cite{ghomi-spruck2023b, ghomi2024}, states that
$$
\G(\Gamma)-\G(\gamma)
=
-\int_{\Omega\setminus D}\sum_{1\leq i\leq n-1}\widehat{GK}_{i}R_{inin}+
\int_{\Omega\setminus D}
\sum_{1\leq i\neq j\leq n-1} \widehat{GK}_{ij} \,\frac{|\nabla u|_j}{|\nabla u|}R_{ijin},
$$
where $|\nabla u|_j:=\nabla_{e_j} |\nabla u|$, $R_{ijk\ell}=\langle R(e_i, e_j)e_k, e_\ell\rangle$ are components of the Riemann curvature of $M$, $\widehat{GK}_{i}$ denotes the product of all principal curvatures other than $\kappa_i$, and $\widehat{GK}_{ij}$ is the product without $\kappa_i$ and $\kappa_j$. Note that $R_{inin}\leq 0$ because these are sectional curvatures of $M$. Furthermore,
if $u$ is a \emph{convex function}, i.e., its composition with geodesics in $M$ is convex, then $\kappa_i\geq 0$. Thus the first integral in  the comparison formula is nonnegative, which immediately yields:

\begin{lemma}\label{lem:convexnested}
Let $\Gamma$, $\gamma$ be $\C^{1,1}$ convex hypersurfaces in $M$, with $\gamma$ nested inside $\Gamma$, and bounding domains $\Omega$, $D$ respectively. Then
\begin{equation*}
\G(\Gamma)-\G(\gamma)
\geq
\int_{\Omega\setminus D}
\sum_{1\leq i\neq j\leq n-1} \widehat{GK}_{ij} \,\frac{|\nabla u|_j}{|\nabla u|}R_{ijin}
\end{equation*}
for any $\C^{1,1}$ convex function $u$ on $\ol \Omega\setminus D$ with $|\nabla u|\neq 0$, which is constant on $\gamma$ and $\Gamma$ with $u|_\gamma<u|_\Gamma$.
\end{lemma}

\subsection{The distance function}

For any set $X\subset M$, the \emph{distance function} $d_X\colon M\to\R$ is defined by
$$
d_X(p):=\inf_{x\in X} \textup{dist}_M(p,x),
$$
where $\textup{dist}_M$ is the Riemannian distance in $M$. For basics of distance functions  see \cite[Sec. 2, 3]{ghomi-spruck2022} and references therein. In particular when $X$ is convex, $d_X$ is convex on $M$. Furthermore,
$d_X$ is locally $\C^{1,1}$ on  $M\setminus \ol X$ \cite[Prop. 2.7]{ghomi-spruck2022} and $|\nabla d_X|=1$. 
A function $f\colon M\to \R$ is locally $\C^{1,1}$ on $X\subset M$ if it is $\C^{1,1}$ in local charts covering $X$. When $f$ is $\C^1$, an equivalent condition is that $\nabla f$ be Lipschitz, i.e.,
$$
\big|\nabla f(p)-\mathcal{T}_{q\to p}\nabla f(q)\big|\leq C \textup{dist}_M(p,q),
$$
for all $p$, $q\in X$, where $\mathcal{T}_{q\to p}$ is parallel translation along the geodesic connecting $q$ to $p$ \cite{azagra-ferrera2015}. Throughout this work $C$ denotes a positive constant whose value may change from one occurrence  to the next. If $f$ is locally $\C^{1,1}$ on a compact set $X$, then we  say that $f$ is $\C^{1,1}$ on $X$. The following fact is known in $\R^n$, see \cite[Thm. 4.8 (5)\&(9)]{federer1959} or \cite[Thm. 6.3]{delfour-zolesio2011}. A set $X\subset M$ is \emph{convex} if it contains the geodesic connecting any pair of its points.

\begin{lemma}\label{lem:d2}
Let $X$ be a convex set in $M$. Then $d_X^2$ is locally $\C^{1,1}$ on $M$.
\end{lemma}
\begin{proof}
Since $d_X=d_{\ol X}$, we may assume that $X$ is closed. Also note that,
since $d_X$ is convex on $M$, and is $\C^1$ on $M\setminus \partial X$, the same holds for $d_X^2$. Let $\log_p\colon M\to T_p M$ be the inverse of the exponential map, $\pi_X\colon M\to X$ be the nearest point projection, and set $\ol p:=\pi_X(p)$. Then
$$
\nabla d_X(p)= -\frac{\log_p(\ol p)}{d_X(p)}
$$
for $p\in M\setminus X$ \cite[Lem. 2.2]{ghomi-spruck2022}.  Let $p_0\in\partial X$, and set $f(p):= \textup{dist}_M^2(p,p_0)$. Since $0\leq d_X^2(p)\leq f(p)$, and $f(p_0)=0$, it follows that $d_X^2$ is differentiable at $p_0$ with $|\nabla d^2_X(p_0)|=0$. Thus $\nabla d^2_X$ is continuous, and so $d_X^2$ is $\C^1$ on $M$.
Since 
$\nabla d_X^2(p)=2 d_X(p)\nabla d_X(p)=-2\log_p(\ol p)$, for $p\in M\setminus X$, and $|\nabla d_X^2|=0$ on $X$,
$$
\nabla d_X^2(p)=-2\log_p(\ol p),
$$
for all $p\in M$. By the triangle inequality,
\begin{multline*}
\left|\nabla d^2_X(p)-\mathcal{T}_{q\to p}\nabla d^2_X(q)\right|
=
2\left|\log_p(\ol p)-\mathcal{T}_{q\to p}\log_q(\ol q)\right|\\
\leq
2\left|\log_p(\ol p)-\log_p(\ol q)\right|
+
2 \left|\log_p(\ol q)-\mathcal{T}_{q\to p}\log_q(\ol q)\right|.
\end{multline*}
Since $K_M\leq 0$, $\log_p$ is nonexpansive. Furthermore, $\pi_X$ is nonexpansive as well \cite[Cor. 2.5]{bridson-haefliger1999}. Thus
$$
\left|\log_p(\ol p)-\log_p(\ol q)\right|\leq \textup{dist}_M(\ol p,\ol q)\leq \textup{dist}_M(p,q).
$$
It now suffices to show that
$\left|\log_p(\ol q)-\mathcal{T}_{q\to p}\log_q(\ol q)\right|\leq C\textup{dist}_M(p,q)$, for $p$, $q$ in any given compact set $Y\subset M$.
More generally, for any fixed point $o$ of $M$, and $p$, $q$ in $Y$ we claim that
$$
\left|\log_p(o)-\mathcal{T}_{q\to p}\log_q(o)\right|\leq  C \textup{dist}_M(p,q),
$$
that is, the vector field $p\mapsto\log_p(o)$ is Lipschitz on $Y$. This is indeed the case because
$$
\log_p(o)=-\frac{1}{2}\nabla\textup{dist}_M^2(p,o),
$$ 
and $\textup{dist}_M^2(p,o)$ is smooth, which completes the proof.
\end{proof}

\subsection{Mixed curvature terms}
The Riemann curvature tensor $R$ may be viewed as a symmetric bilinear form $\mathcal{R}$ on the space of 2-forms $\Lambda^2 TM$. More explicitly, let $e_i$ be an orthonormal basis for $T_p M$. Then $e_i\wedge e_j$, for $1\leq i<j\leq n$, form a basis for $\Lambda^2 T_pM$. There is a natural inner product on $\Lambda^2 T_pM$ given by $\langle e_i \wedge e_j,\, e_k \wedge e_\ell \rangle := \delta_{ik} \delta_{j\ell} - \delta_{i\ell} \delta_{jk}$. In particular, $e_i\wedge e_j$ are orthonormal.
We may then define $\mathcal{R}\colon \Lambda^2 T_pM\to \Lambda^2 T_pM$ by 
$$
\big\langle \mathcal{R}(e_i\wedge e_j), e_k\wedge e_\ell\big\rangle:= R_{ijk\ell}=\big\langle R(e_i, e_j)e_k,e_\ell\big\rangle.
$$
The \emph{mixed curvature terms} are the coefficients $R_{ijk\ell}$  when $\{i,j\}\neq \{k,\ell\}$, or the off-diagonal components of  $\mathcal{R}$. We say $K$  is constant on $X\subset M$, or $K=k$ on $X$, if for all $p\in X$ and planes $\Pi\subset T_pM$, $K(\Pi)=k$. 

\begin{lemma}\label{lem:mixed}
Let $X\subset M$ be a compact set. Suppose that $K$ is constant on $X$. Then there exists a neighborhood $U$ of $X$ such that for any orthonormal frame field on $U$, the absolute values of  the mixed curvature terms on $U$ are bounded above by $C d_X$.
\end{lemma}
\begin{proof}
Suppose that $K=k$ on $X$, and let $e_i$ be a smooth orthonormal frame field on a neighborhood $V$ of $X$. Let $\tilde{\mathcal{R}}$ be the matrix representation of $\mathcal{R}$ with respect to $e_i\wedge e_j$.
Then $\tilde{\mathcal{R}}=kI$ on $X$, where  $I$ is the identity matrix. Since $X$ is compact and $\tilde{\mathcal{R}}$ is smooth, it follows that
$$
|\tilde{\mathcal{R}}-kI|_\infty\leq C d_X,
$$
on a neighborhood $U\subset V$, where $|\cdot|_\infty$ is the supremum of the absolute values of the coefficients. More explicitly, the above inequality follows from applying the mean value theorem to  $\tilde{\mathcal{R}}-kI$ restricted to geodesic segments originating from points of $X$.
If $\tilde{\mathcal{R}}'$ is the matrix representation of $\mathcal{R}$ with respect to $e_i'\wedge e_j'$, for any other frame field $e_i'$, then $\tilde{\mathcal{R}}' =O^{T}\tilde{\mathcal{R}}O$ for an orthogonal matrix $O$ at each point.
Thus
$$
\big|\tilde{\mathcal{R}}'-kI\big|_\infty=\big|O^{T}(\tilde{\mathcal{R}}-kI)O\big|_\infty\leq C \big|\tilde{\mathcal{R}}-kI\big|_\infty,
$$
where $C$ depends only on $n$. So $|\tilde{\mathcal{R}}'-kI|_\infty\leq C d_X$, which completes the proof.
\end{proof}

\subsection{Continuity of total curvature}
For a general convex hypersurface $\Gamma\subset M$, the total curvature $\mathcal{G}(\Gamma)$ is defined as follows. Let $\Omega$ be the domain bounded by $\Gamma$. The \emph{outer parallel hypersurface} of $\Gamma$  at distance $t > 0$ is given by 
$
\Gamma_t:=d_\Omega^{-1}(t).
$
 For $t>0$, $\Gamma_t$ is $\C^{1,1}$ \cite[Lem. 2.6]{ghomi-spruck2022} and thus $\G(\Gamma_t)$ is well defined. We set 
$$
\G(\Gamma):=\lim_{t\searrow 0}\G(\Gamma_t).
$$
By the comparison formula, 
$t\mapsto \G(\Gamma_t)$ is nondecreasing \cite[Cor. 4.4]{ghomi-spruck2023b}. Furthermore, since $\Gamma_t$ is convex, $\G(\Gamma_t)\geq 0$. Thus $\G(\Gamma)$ is well-defined and finite. We record the following known fact \cite[Note 3.7]{ghomi-spruck2022}, which can be established via the theory of smooth valuations \cite{alesker2007},   and convergence of normal cycles \cite{fu1994,zahle1986}. See \cite{ghomi2026} for a more direct proof.

\begin{lemma}[\cite{ghomi2026}]\label{lem:continuity}
The total curvature functional $\mathcal{G}$ is continuous with respect to Hausdorff distance on the space of convex hypersurfaces $\Gamma\subset M$.
\end{lemma}

\section{Proof of Theorem \ref{thm:main}}
\subsection{The general case} 
Let $\Omega$, $D$ be the domains bounded by $\Gamma$, $\gamma$, and $d_\Omega$, $d_D$ be the corresponding distance functions respectively.
For $\lambda\in [0,\lambda_0]$, set
$$
u^\lambda:=\lambda d_D+d_\Omega^2.
$$
Recall that $d_D$ is locally $\C^{1,1}$ on $M\setminus\ol D$. Furthermore, $d_\Omega^2$ is locally $\C^{1,1}$ on $M$ by Lemma \ref{lem:d2}. Thus $u^\lambda$ is  locally $\C^{1,1}$ on $M\setminus\ol D$. Since $d_D$ and $d_\Omega$ are convex, so is $u^\lambda$. Fix $\epsilon>0$ so small that the outer parallel hypersurface $\gamma_\epsilon$ is nested inside $\Gamma$. 
Let $D_\epsilon$ and $\Omega_\epsilon$ be the domains bounded by $\gamma_\epsilon$ and $\Gamma_\epsilon$ respectively.
Set
$$
\Gamma_\epsilon^\lambda:=(u^{\lambda})^{-1}(\epsilon^2).
$$
 So $\Gamma_\epsilon^\lambda\to\Gamma_\epsilon$ as $\lambda\to 0$.
In particular, choosing $\lambda_0$ sufficiently small,
we may assume that $\Gamma_\epsilon^\lambda$ lies in
the annular region $\Omega_{2\epsilon}\setminus\Omega$.
Hence if $\Omega_\epsilon^\lambda$ is the domain bounded by
$\Gamma_\epsilon^\lambda$, then
$\Omega\subset\Omega_\epsilon^\lambda\subset\Omega_{2\epsilon}$.
Furthermore, $|\nabla u^\lambda|_j=0$ on
$\Omega\setminus D_\epsilon$, since
$u^\lambda=\lambda d_D$ there and $|\nabla d_D|=1$.
We may choose $\epsilon$ so small that $K$ is constant on
$\Omega_{2\epsilon}\!\!\setminus\Omega$. Then the mixed
curvature terms $R_{ijin}=0$ on
$\Omega_{2\epsilon}\!\!\setminus\Omega$.
Hence $\G(\Gamma_\epsilon^\lambda)-\G(\gamma_\epsilon)\geq 0$
by Lemma \ref{lem:convexnested}.
Letting $\lambda\to 0$ followed by $\epsilon\to 0$
completes the proof by Lemma \ref{lem:continuity}.

\subsection{The case of $n=3$}
When $n=3$, by Lemma \ref{lem:convexnested} we have
$$
\G(\Gamma_\epsilon^\lambda)-\G(\gamma_\epsilon)
\geq
\int_{\Omega_\epsilon^\lambda\setminus D_\epsilon} F_\lambda,
\quad\quad
\text{where}
\quad\quad
F_\lambda
:=
\sum_{1\leq i,j\leq 2} \frac{|\nabla u^\lambda|_{j}}{|\nabla u^\lambda|}R_{ijin}.
$$
Since $\Omega_\epsilon^\lambda\setminus D_\epsilon=(\Omega_\epsilon^\lambda\setminus \Omega)\cup (\Omega\setminus D_\epsilon)$, $u^\lambda=\lambda d_D$ on $\Omega\setminus D_\epsilon$, and $|\nabla d_D|=1$, it follows that
$F_\lambda$ vanishes identically on $\Omega\setminus D_\epsilon$.
Thus 
$$
\G(\Gamma_\epsilon^\lambda)-\G(\gamma_\epsilon)
\geq
\int_{\Omega_\epsilon^\lambda\setminus \Omega} F_\lambda
\geq
-\left|\int_{\Omega_{\epsilon}^\lambda\setminus \Omega} F_\lambda\right|
\geq
-\int_{\Omega_{\epsilon}^\lambda\setminus \Omega} |F_\lambda|
\geq
-\int_{\Omega_{2\epsilon}\!\setminus\Omega} \left|F_\lambda\right|.
$$

Next we show that $|F_\lambda|$ is uniformly bounded above (almost everywhere) on $\Omega_{2\epsilon}\!\setminus\Omega$, by the following three estimates.
Since  $\nabla u^\lambda=\lambda \nabla d_D +\nabla d_\Omega^2$ is uniformly Lipschitz, 
$$
\big||\nabla u^\lambda|_j\big|\leq C,
$$
for $C$ independent of $\lambda$. By Lemma \ref{lem:mixed}, we also have
$$
|R_{ijin}|\leq Cd_\Omega,
$$
where again $C$ does not depend on $\lambda$.
Next note that
$
\langle\nabla d_\Omega,\nabla  d_D \rangle \geq 0,
$
because level sets of $d_\Omega$ are convex, and $\nabla  d_D$ is tangent to geodesic rays which originate in $\Omega$. Thus
\begin{equation}\label{eq:nablau}
\big|\nabla u^{\lambda }\big| =\sqrt{4d_\Omega^{2}+\lambda ^{2}+4\lambda d_\Omega\langle\nabla d_\Omega,\nabla  d_D \rangle}\geq 2d_\Omega.
\end{equation}
So we conclude that $|F_\lambda|\leq C$ on $\Omega_{2\epsilon}\!\setminus\Omega$, which yields
$$
\G(\Gamma_\epsilon^\lambda)-\G(\gamma_\epsilon)
\geq
-C|\Omega_{2\epsilon}\!\setminus\Omega|.
$$
Again letting $\lambda\to 0$ followed by $\epsilon\to 0$ completes the proof by Lemma \ref{lem:continuity}.

\subsection{The $\C^{1,1}$ case}\label{subsec:c11-monotonicity}
Finally suppose that $\Gamma$ is $\C^{1,1}$ and $K$ is constant on $\Gamma$.
We use the construction above with $\lambda>0$ and the principal frame
of the comparison formula. Set
$$
F_\lambda:=\sum_{1\leq i\ne j\leq n-1}
\widehat{GK}_{ij}
\frac{|\nabla u^\lambda|_j}{|\nabla u^\lambda|}R_{ijin}.
$$
As in the preceding subsection, Lemma~\ref{lem:convexnested} reduces
the proof to showing that $|F_\lambda|\leq C$ on
$\Omega_{2\epsilon}\setminus\Omega$, with $C$ independent of
sufficiently small $\lambda$ and $\epsilon$.

The signed distance to $\Gamma$ is $\C^{1,1}$ near $\Gamma$
\cite[Lem. 2.6]{ghomi-spruck2022}. Thus $\nabla^2d_\Omega$ is bounded
on $\Omega_{2\epsilon}\setminus\Omega$, uniformly for all sufficiently
small $\epsilon$. The same holds for $\nabla^2d_D$, since this region
stays a positive distance from $\ol D$.
All subsequent estimates are on this region and hold almost everywhere. Using \eqref{eq:nablau}, we have
\begin{equation}\label{eq:c11-gradient}
\frac{\lambda+2d_\Omega}{\sqrt2}
\leq|\nabla u^\lambda|
\leq\lambda+2d_\Omega\leq C.
\end{equation}
Furthermore, we compute that 
\begin{equation}\label{eq:c11-hessian-decomposition}
\nabla^2u^\lambda
=2\nabla d_\Omega\otimes\nabla d_\Omega+|\nabla u^\lambda|A\qquad \text{where}\qquad A:=\frac{\lambda\nabla^2d_D+2d_\Omega\nabla^2d_\Omega}
{|\nabla u^\lambda|}.
\end{equation}
The Hessian bounds and \eqref{eq:c11-gradient} give $|A|\leq C$.
Let $w:=\nabla d_\Omega-\langle\nabla d_\Omega,e_n\rangle e_n$,
and $S$ be the shape operator of the level sets of $u^\lambda$ with respect to $e_n$.
For every tangent vector $v$, \eqref{eq:c11-hessian-decomposition} gives
$$
\langle Sv,v\rangle
=\frac{2\langle w,v\rangle^2}{|\nabla u^\lambda|}+A(v,v).
$$
Thus $\langle Sv,v\rangle\leq C|v|^2$ for tangent vectors $v\perp w$.
Since $w^\perp$ has codimension at most $1$, at most one principal
curvature exceeds $C$. Ordering the nonnegative principal curvatures
so that $\kappa_1\geq\cdots\geq\kappa_{n-1}$, we therefore have
$\kappa_i\leq C$ for $i\geq2$.
Furthermore, $|w|\leq1$ and \eqref{eq:c11-gradient} give
$\kappa_1\leq\|S\|\leq2/|\nabla u^\lambda|+C\leq C'/|\nabla u^\lambda|$.
Set $w_i:=\langle\nabla d_\Omega,e_i\rangle$ and $A_{ij}:=A(e_i,e_j)$.
So \eqref{eq:c11-hessian-decomposition} gives
$$
\kappa_i=\frac{2w_i^2}{|\nabla u^\lambda|}+A_{ii},\qquad
2w_iw_j+|\nabla u^\lambda|A_{ij}=0\qquad(i\ne j).
$$
For $j\ne1$, multiplying the first identity with $i=1$ by $w_j$
and using the second gives
$$
\kappa_1w_j
=\frac{2w_1^2w_j}{|\nabla u^\lambda|}+A_{11}w_j
=-w_1A_{1j}+A_{11}w_j.
$$
Thus $|\kappa_1w_j|\leq C$, since $|w_i|\leq1$ and $A$ is bounded.
Together with the bounds for the other principal curvatures, this yields
\begin{equation}\label{eq:inequalities}
|\widehat{GK}_{ij}w_j|\leq C,\qquad
|\widehat{GK}_{ij}|\leq\frac{C}{|\nabla u^\lambda|},
\end{equation}
for $i\ne j$.
Indeed, if $\kappa_1$ occurs in $\widehat{GK}_{ij}$, then $j\ne1$,
so combining that factor with $w_j$ proves the first inequality.
If $\kappa_1$ is omitted, all factors are bounded.
For the second inequality, the curvature bounds and
\eqref{eq:c11-gradient} give
$|\widehat{GK}_{ij}|\leq C(1+\kappa_1)\leq C/|\nabla u^\lambda|$.
Finally,  using
\eqref{eq:c11-hessian-decomposition} we have
$$
|\nabla u^\lambda|_j
=\nabla^2u^\lambda(e_n,e_j)
=2\langle\nabla d_\Omega,e_n\rangle w_j
+|\nabla u^\lambda|A(e_n,e_j).
$$
The two inequalities in \eqref{eq:inequalities} therefore yield
$$
\left|\widehat{GK}_{ij}
\frac{|\nabla u^\lambda|_j}{|\nabla u^\lambda|}\right|
\leq\frac{C}{|\nabla u^\lambda|}.
$$
Recall also that $|R_{ijin}|\leq Cd_\Omega$, by
Lemma~\ref{lem:mixed}. Summing the resulting bounds and using
\eqref{eq:c11-gradient}, we obtain
$|F_\lambda|\leq Cd_\Omega/|\nabla u^\lambda|\leq C$, as desired.

\section*{Acknowledgment}
We thank Mario Santilli for comments on proving Lemma \ref{lem:d2}. Thanks also to Joel Spruck and Joe Hoisington for  useful communications.

\bibliography{references}

\end{document}